\documentclass[12pt]{amsart}
\usepackage[cp1251]{inputenc} 
\usepackage[english]{babel}
\usepackage{amsthm,amssymb,amsmath,amsfonts,mathrsfs, mathtools}
\usepackage{amsmath}
\usepackage{geometry}
\makeatletter
\newcommand*{\rom}[1]{\expandafter\@slowromancap\romannumeral #1@}
\makeatother

\makeatletter
\@addtoreset{equation}{section}
\makeatother

\newtheorem{theorem}[equation]{Theorem}
\newtheorem{proposition}[equation]{Proposition}
\newtheorem{lemma}[equation]{Lemma}
\newtheorem{corollary}[equation]{Corollary}

\theoremstyle{definition}
\newtheorem{example}[equation]{Example}
\newtheorem{definition}[equation]{Definition}

\theoremstyle{remark}
\newtheorem{remark}[equation]{Remark}

\def \Hom {\operatorname{Hom}}
\def \End {\operatorname{End}}

\def \ind {\operatorname{ind}}

\def \SL {\operatorname{SL}}

\def \Spec {\operatorname{Spec}}

\newcommand{\CC}{\mathbb{C}}
\newcommand{\ZZ}{\mathbb{Z}}

\newcommand{\QQ}{\mathbb{Q}}

\newcommand{\GG}{\mathbb{G}}

\title{Irreducible representations of finitely generated nilpotent groups}

\author{Iuliya Beloshapka}
\address{Department of Mechanics and Mathematics, Moscow State University, Russia}
\email{i-beloshapka@yandex.ru}

\author{Sergey Gorchinskiy}
\address{Steklov Mathematical Institute of Russian Academy of Sciences, Moscow, Russia}
\address{National Research University Higher School of Economics, Moscow, Russia}
\email{gorchins@mi.ras.ru}

\date{}

\begin{document}

\maketitle

\begin{center}
{\it \small To our teacher Alexey Nikolaevich Parshin, with admiration}
\end{center}

\begin{abstract}
We prove that irreducible complex representations of finitely generated nilpotent groups are monomial if and only if they have finite weight, which was conjectured by Parshin. Note that we consider (possibly, infinite-dimensional) representations without any topological structure. Besides, we prove that for certain induced representations, irreducibility is implied by Schur irreducibility. Both results are obtained in a more general form for representations over an arbitrary field.
\end{abstract}

\section{Introduction}

It is a classical result that irreducible complex representations of finite nilpotent groups are monomial, that is, are induced from characters of subgroups (see, e.g.,~\cite[\S~8.5, Theorem~16]{Serre}). Kirillov~\cite{Kir59} (see also~\cite[Theorem 5.1]{Kirillov}) and Dixmier~\cite[Th\'eor\`eme 2]{Dix} have independently proved an analogous statement for irreducible unitary representations of connected nilpotent Lie groups.

Later, Brown~\cite{Brown} claimed that irreducible unitary representations of (discrete) finitely generated nilpotent groups are monomial if and only if they have finite weight. Recall that a representation $\pi$ of a group~$G$ has finite weight if there is a subgroup $H\subset G$ and a character $\chi$ of $H$ such that the vector space $\Hom_H(\chi,\pi\vert_H)$ is non-zero and finite-dimensional. 

In the plenary lecture at ICM2010, Parshin~\cite[\S~5.4(i)]{ParshCong} (see also~\cite[page~296]{Parshin-Arnal}) conjectured that Brown's equivalence holds for all irreducible complex representations of finitely generated nilpotent groups, without any topological structure on representations. In this setting, by a monomial representation, one means a finitely induced representation (see Definition~\ref{definition:compact}) from a character of a subgroup.

\medskip

Parshin's conjecture is known to be true in some particular cases. Firstly, a similar argument as for finite nilpotent groups shows that all finite-dimensional irreducible complex representations of finitely generated nilpotent groups are monomial (see, e.g.,~\cite[Lemma 1]{Brown} or Proposition~\ref{prop:indfin}).

Secondly, for finitely generated abelian groups, the conjecture holds true, because all irreducible representations of such groups are just characters (this follows from a generalization of Schur's lemma, see, e.g.,~\cite[Claim 2.11]{BZ} or Proposition~\ref{prop:schur}). For the next case of finitely generated nilpotent groups of nilpotency class two, the conjecture was proved by Arnal and Parshin~\cite{Parshin-Arnal}.

Finally, it is easy to show one implication in the conjecture: if an irreducible complex representation is monomial, then it has finite weight (see Proposition~\ref{prop:weight}(ii)).

\medskip

We prove Parshin's conjecture in full generality, which is the main result of the paper (see Theorem~\ref{theorem:main} and also a specification in Remark~\ref{remark:maxconj}).

\medskip

\begin{itemize}
\item[]{}
{\it {\sc Theorem A.} Let $G$ be a finitely generated nilpotent group and $\pi$ a (possibly, infinite-dimensional) irreducible complex representation of~$G$. Then $\pi$ is monomial if and only if $\pi$ has finite weight.
}
\end{itemize}

\medskip

In fact, we prove a more general result on representations over an arbitrary field, which may be non-algebraically closed and may have a positive characteristic (see Theorem~\ref{theorem:key}).

\medskip

\begin{itemize}
\item[]{}
{\it {\sc Theorem B.} Let $G$ be a finitely generated nilpotent group and $\pi$ an irreducible representation of~$G$ over an arbitrary field $K$. Suppose that there is a subgroup $H'\subset G$ and a finite-dimensional irreducible representation $\rho'$ of $H'$ over $K$ such that the vector space $\Hom_{H'}(\rho,\pi|_{H'})$ is non-zero and finite-dimensional. Then there is a subgroup~$H\subset G$ and a finite-dimensional irreducible representation~$\rho$ of $H$ over $K$ such that~$\pi$ is isomorphic to the finitely induced representation $\ind_{H}^G(\rho)$.
}
\end{itemize}

\medskip

Notice that, in general, the pairs $(H,\rho)$ and $(H',\rho')$ as in Theorem B are different. Theorem B implies directly Theorem~A (see Subsection~\ref{subsect:stat}).

\medskip

Theorem A (see also Proposition~\ref{prop:twosubgroups}) can be applied to a description of the moduli space of irreducible representations of finitely generated nilpotent groups. In the case of nilpotency class two, this was done by Parshin~\cite{Parshin}.

Moduli spaces of representations of finitely generated nilpotent groups naturally arise in the study of algebraic varieties by methods of higher-dimensional adeles. These moduli spaces are expected to be used in questions related to \mbox{$L$-functions} of varieties over finite fields, see more details in~\cite{ParshCong}.

Another motivation to study representations without a topological structure and to construct their moduli spaces is Bernstein's theory of smooth complex representations of reductive $p$-adic groups (see, e.g.~\cite{Bernst}).

\medskip

Note that there are irreducible complex representations of finitely generated nilpotent groups that do not satisfy the equivalent conditions of Theorem~A. The examples were constructed by Brown~\cite[\S~2]{Brown} in the context of unitary representations and independently by Berman, \v{S}araja~\cite{BS} and Segal~\cite[Theorems~A, B]{Seg} for representations without a topological structure. A detailed analysis of non-monomial representations for the Heisenberg group over the ring of integers was made by Berman and Kerer~\cite{BK}.

\medskip

A sharp distinction between Brown's setting and Theorem~A is that Brown treats unitary representations, while Theorem~A concerns complex representations without any topological structure. This leads to numerous differences, most notably, the following one. The category of unitary representations is semi-simple. On the other hand, there are non-trivial extensions between representations without a topological structure and, in general, the converse to Schur's lemma does not hold for such representations (see Example~\ref{examp:contr} and the example in Subsection~\ref{subsect:Heis}).

\medskip

Our proof of Theorem~B is based on several crucial ideas from~\cite{Brown}, in particular, we use a certain group-theoretic result on nilpotent groups (see Proposition~\ref{prop:Brownnorm}). Following Brown, we modify the pair $(H',\rho')$ as in Theorem B in order to get the pair $(H,\rho)$. Unfortunately, one of the steps in Brown's strategy of modification is based on a false statement, namely,~\cite[Lemma 6]{Brown} (see Remark~\ref{remark:mistake}).

Thus we have changed the strategy. A surprising phenomenon is that, while constructing the pair~${(H,\rho)}$ as above, we pass through auxiliary pairs $(H_0,\rho_0)$ such that the vector space~${\Hom_{H_0}(\rho_0,\pi|_{H_0})}$ is non-zero but, possibly, has infinite dimension. However, these pairs do satisfy another finiteness condition, namely, they are so-called perfect pairs (see Definition~\ref{def:s}(ii)).

We believe that our strategy of the proof of Theorem~B can be also applied to obtain a correct proof of Brown's equivalence for unitary representations.


\medskip

Another essential new ingredient of the proof of Theorem~B is the following result, which is of independent interest for representation theory: the converse to Schur's lemma does hold true for finitely induced representations from irreducible representations of normal subgroups (see Proposition~\ref{lemma}; for simplicity, we state it here for the case of complex representations).

\medskip

\begin{itemize}
\item[]{}
{\it {\sc Proposition.} Let $H$ be a normal subgroup of an arbitrary group~$G$. Let $\rho$ be an irreducible complex representation of $H$ such that the finitely induced representation $\ind_{H}^{G} (\rho)$ satisfies $\End_{G} \big( \ind_{H}^{G} (\rho) \big) = \CC$. Then the representation~${\ind_{H}^{G} (\rho)}$ is irreducible.
}
\end{itemize}

\medskip

Note that irreducibility of induced representations of connected Lie groups was studied in detail by Jacobsen and Stetk{\ae}r~\cite{JS}.

\medskip

The paper is organized as follows. In Section~\ref{sect:prel}, we provide mostly known results that are used later in the proof of the main theorem. Subsection~\ref{subsect:not} introduces the notation that is used throughout the paper. In Subsection~\ref{sect:roots}, we make a modification of a group-theoretic result of Brown~\cite[Lemma 4]{Brown} suitable for our needs (see Theorem~\ref{theor:S}). Subsection~\ref{sect:end} collects well-known formulas for endomorphisms of finitely induced representations (see Proposition~\ref{prop:end} and Corollary~\ref{lemma:index}), based on Frobenius reciprocity and Mackey's formula. In Subsection~\ref{subsect:Brown}, we define the notion of a $\pi$-irreducible pair for a representation $\pi$ (see Definitions~\ref{definition:weight}(i) and~\ref{defin:pipair}(i)), which is our main tool to show that a representation is finitely induced. We also prove a result that allows to extend $\pi$-irreducible pairs (see Lemma~\ref{lemma:extend}).

Section~\ref{sect:irred} is devoted to irreducibility of finitely induced representations. In Subsection~\ref{sect:irr}, we prove that Schur irreducibility implies irreducibility for certain induced representations (see Proposition~\ref{lemma}, Remark~\ref{remark:equivSchur}, and Corollary~\ref{theorem:irr}). We apply this in Subsection~\ref{sect:nilp} to representations of finitely generated nilpotent groups, obtaining a sufficient condition for irreducibility of finitely induced representations (see Theorem~\ref{theor:char}). In Subsection~\ref{subsect:Heis}, we construct an example showing that, in general, Schur irreducibility does not imply irreducibility for representations of finitely generated nilpotent groups. The example concerns the simplest nilpotent group which is not abelian-by-finite, namely, the Heisenberg group over the ring of integers.

In Section~\ref{sect:main}, we state and prove the main results of the paper. In Subsection~\ref{subsect:stat}, we formulate our key result (see Theorem~\ref{theorem:key}) and deduce from it the equivalence for monomial and finite weight representations (see Proposition~\ref{prop:weight} and Theorem~\ref{theorem:main}). Subsection~\ref{subsect:proof} consists in the proof of Theorem~\ref{theorem:key}. We provide in Subsection~\ref{subsect:isommonom} an isomorphism criterion for finitely induced representations (see Proposition~\ref{prop:twosubgroups}), which repeats essentially~\cite[Theorem 2]{Brown}. Finally, in Subsection~\ref{subsect:nonmonom}, following~\cite{BS} and~\cite{Seg}, we provide an example of an irreducible complex representation of the Heisenberg group over the ring of integers which is not finitely induced from a representation of a proper subgroup.

\medskip

During the work on the paper, we learned from A.\,N.\,Parshin that E.\,K.\,Narayanan and P.\,Singla are studying independently the same subject.

\medskip

We are deeply grateful to A.\,N.\,Parshin for posing the problem and a constant attention to the progress. It is our pleasure to thank C.\,Shramov for many discussions that were highly valuable and stimulating. We are grateful to S.\,Nemirovski for drawing our attention to the paper~\cite{JS}. Both authors were supported by the grants MK-5215.2015.1, NSh-2998.2014.1, and RFBR 14-01-00178. The first named author acknowledges the support of the grants RFBR 14-01-00160 and 13-01-00622. The second named author acknowledges the support of the grant RFBR 13-01-12420, the grant of Dmitry Zimin's Foundation ``Dynasty'', and the subsidy granted to the HSE by the Government of the Russian Federation for the implementation of the Global Competitiveness Program. The second named author is also very grateful for hospitality and excellent working conditions to Institut de Math\'ematiques de Jussieu, where a part of the work was done.

\bigskip

\section{Preliminaries}\label{sect:prel}

\subsection{Notation}\label{subsect:not}

We fix a field $K$ (a priori we do not make additional assumptions on~$K$). For short, by a vector space, we mean a (possibly, infinite-dimensional) vector space over $K$. By a representation of a group, we mean a (possibly, infinite-dimensional) representation over~$K$.

Throughout the paper, $G$ denotes a group and $H$ a subgroup of $G$. Given a subset $E\subset G$, by $\langle E\rangle$ denote the subgroup of~$G$ generated by~$E$.

Further, $\pi$ denotes a representation of~$G$, $\rho$ a representation of $H$, and $\chi\colon H\to K^*$ a character of~$H$. By~${\pi|_H}$ denote the restriction of $\pi$ to~$H$.

For an element $g \in G$, let~${H^{g} \subset G}$ be the conjugate subgroup~${H^{g}=g H g^{-1}}$ and~$\rho^{g}$ the representation of $H^{g}$ defined by the formula $\rho^g(ghg^{-1})=\rho(h)$, where $h \in H$.

\medskip

We mention it explicitly if we require some more properties of the field $K$, groups, or representations.

\subsection{A result from group theory}\label{sect:roots}

By $N_{G}(H)$ denote the normalizer of $H$ in $G$.

\begin{definition} \label{def:ss}
Let $S(H) \subset G$ be the set of all elements $g \in G$ such that the index of $H^g \cap H$ in $H$ is finite.
\end{definition}

Clearly, there is an embedding $N_G(H)\subset S(H)$.

\begin{example}\label{exam:s}
Let $G$ be the group $\SL_2(\ZZ)$ and $H$ the subgroup of all matrices whose lower left entry equals zero. Then a direct calculation shows that~${S(H)=H}$.
\end{example}

\medskip

The following construction will allow us to give an upper bound on the set~${S(H)}$ (see Lemma~\ref{lemma:normalizer} below).

\begin{definition}\label{def:star}
Let $H^{*}$ be the smallest subgroup of $G$ with the following properties: $H^*$ contains $H$ and if an element $g\in G$ satisfies $g^{i} \in H^*$ for some positive integer $i$, then~${g\in H^{*}}$.
\end{definition}

It is easily shown that $H^{*}$ is well-defined, that is, $H^*$ exists (and is unique) for any subgroup $H\subset G$.

\begin{remark}\label{remark:trivial}
\hspace{0cm}
\begin{itemize}
\item[(i)]
There is an equality $(H^*)^*=H^*$.
\item[(ii)]
For any element $g \in G$, we have $(H^{g})^{*} = (H^{*})^{g} $ (cf.~\cite[Lemma 4(1)]{Brown}).
\end{itemize}
\end{remark}

Recall that a group is called {\it Noetherian} if any increasing chain of its subgroups stabilizes. Obviously, this is equivalent to the fact that any subgroup is finitely generated.

\begin{lemma} \label{lemma:normalizer}
Suppose that $G$ is Noetherian. Then there is an embedding~${S(H) \subset  N_{G}(H^{*})}$.
\end{lemma}
\begin{proof}
Consider an element $g \in S(H)$. By definition, the index of $H^g \cap H$ in~$H$ is finite. Hence there is a positive integer~$i$ such that for any element~${h \in H}$, we have~${h^i\in H^g}$. Therefore, $H\subset (H^g)^{*}$. By Remark~\ref{remark:trivial}, we see that $H^*\subset (H^*)^g$. Applying conjugation by positive powers of $g$, we obtain an increasing chain of subgroups
$$
H^*\subset (H^*)^g\subset\ldots\subset (H^*)^{g^i} \subset (H^*)^{g^{i+1}}\subset\ldots
$$
Since $G$ is Noetherian, the chain stabilizes. This implies that ${H^*=(H^*)^g}$, that is, $g\in N_{G}(H^{*})$.
\end{proof}

The following example shows that Lemma~\ref{lemma:normalizer} does not hold for an arbitrary group $G$.

\begin{example}\label{examp:infinind}
Let $G$ be the free group generated by elements $x$ and $y$. Let~$H$ be the subgroup of $G$ generated by the elements $x^{-n}yx^{n}$, where $n$ runs over all positive integers. One easily shows that $H$ is freely generated by the elements~$x^{-n}yx^{n}$, thus $G$ is not Noetherian. Since $H\subset H^x$, we have $x\in S(H)$ and $H^x\cap H=H$. However $H^x$ contains the element $y$, which does not belong to $H^*$ (actually, we have $H=H^*$). Consequently,~${x\notin N_G(H^*)}$.
\end{example}

\medskip

Until the end of this subsection, we suppose that the group $G$ is {\it finitely generated and nilpotent}. It turns out that much more can be said about $S(H)$ in this case. The following crucial result was essentially obtained by Malcev (see a comment to the proof of~\cite[Theorem 8]{Mal}); a complete proof can be found, e.g., in~\cite[Lemma 2.8]{Baumslag}.

\begin{proposition}\label{lemma:finite-index}
The index of $H$ in $H^{*}$ is finite.
\end{proposition}

In other words, Proposition~\ref{lemma:finite-index} claims that $H^*$ is the largest subgroup of $G$ that contains $H$ as a subgroup of finite index. Equivalently, $H^{*}$ coincides with the set of all roots of elements of $H$.

\begin{remark}\label{remark:intersect}
Using Proposition~\ref{lemma:finite-index}, one shows easily that there is an equality ${(H_1\cap H_2)^*=H_1^*\cap H_2^*}$ for all subgroups $H_1,H_2\subset G$ (cf.~\cite[Lemma 4(2)]{Brown}).
\end{remark}

Using Proposition~\ref{lemma:finite-index}, Brown~\cite[Lemma 4(3),(4)]{Brown} has shown the following fact.

\begin{proposition}\label{prop:Brownnorm}
There is an equality $N_G(H^*)=N_G(H)^*$ and this subgroup of~$G$ coincides with the set of all elements $g\in G$ such that
the indices of~${H^g\cap H}$ in both $H$ and $H^g$ are finite.
\end{proposition}

\medskip

Combining Lemma~\ref{lemma:normalizer} with Proposition~\ref{prop:Brownnorm}, we obtain the following useful result.

\begin{theorem}\label{theor:S}
Suppose that the group $G$ is finitely generated and nilpotent. Then the following holds true:
\begin{itemize}
\item[(i)]
the subset $S(H)\subset G$ is a subgroup;
\item[(ii)]
the index of $N_G(H)$ in $S(H)$ is finite;
\item[(iii)]
for any finite index subgroup $H'\subset H$, we have $S(H')=S(H)$.
\end{itemize}
\end{theorem}
\begin{proof}
Recall that any finitely generated nilpotent group is Noetherian~\cite[Theorem~2.18]{Man}. Thus Lemma~\ref{lemma:normalizer} implies the embedding~${S(H)\subset N_G(H^*)}$. By Propositions~\ref{prop:Brownnorm}, we have an opposite embedding, whence $S(H)$ coincides with the subgroup~${N_G(H^*)=N_G(H)^*}$, which proves item~(i). By Proposition~\ref{lemma:finite-index}, the index of $N_G(H)$ in~${S(H)=N_G(H)^*}$ is finite, which is item~(ii). If the index of~$H'$ in $H$ is finite, then there is an equality $(H')^*=H^*$. This implies item~(iii), because, as shown above, $S(H')=N_G\big((H')^*\big)$ and $S(H)=N_G(H^*)$.
\end{proof}

\subsection{Endomorphisms of finitely induced representations} \label{sect:end}

Recall that~$\rho$ is a representation of a subgroup $H\subset G$. By $V$ denote the representation space of $\rho$. Let $V\times_H G$ be the quotient set of~${V\times G}$ by the diagonal action of $H$ given by the formula
$$
{h(v,g)=\big(\rho(h)v,hg\big)}\,.
$$
We have a natural map
$$
p\;:\; V\times_H G\longrightarrow H\backslash G
$$
to the set of right cosets of $H$ in $G$. Note that one has (right) actions of~$G$ on both $V\times_H G$ and $H\backslash G$ by right translations and the map $p$ commutes with these actions. Thus one can say that $V\times_H G$ is a ``$G$-equivariant discrete vector bundle'' on $H\backslash G$.

\begin{definition}\label{definition:compact}
A {\it finitely induced representation} $\ind_H^G(\rho)$ is the representation of $G$ whose representation space consists of all sections of the map $p$ that have finite support on $H\backslash G$. Right translations by $G$ define the action of $G$ on this space.
\end{definition}

\medskip

By {\it Frobenius reciprocity} (see, e.g.,~\cite[Chapter I, \S 5.7]{Vigneras}), for any representation $\pi$ of $G$, there is a canonical isomorphism of vector spaces
\begin{equation}\label{eq:Frob1}
\Hom_{G} \big(\ind_{H}^{G} (\rho), \pi \big) \simeq \Hom_{H} \big( \rho, \pi \vert_{H} \big) \, .
\end{equation}

If the index of $H$ in $G$ is finite, then there is also a canonical isomorphism of vector spaces
\begin{equation}\label{eq:Frob2}
\Hom_{G} \big( \pi, \ind_{H}^{G} (\rho) \big) \simeq \Hom_{H} \big( \pi \vert_{H}, \rho\big) \, .
\end{equation}
Indeed, a natural analog of the isomorphism~\eqref{eq:Frob2} holds true for induced representations constructed similarly as in Definition~\ref{definition:compact} but without the finiteness condition on supports of sections (see, e.g.,~\cite[Chapter~I, \S~5.4]{Vigneras}). When the index of $H$ in $G$ is finite, the latter induction coincides with the finite induction.

\medskip

Given an element $g \in G$, by $\bar g \in H \backslash G /  H$ denote the corresponding double coset $HgH$. Note that the representation~${\ind^{H}_{H^{g} \cap H} (\rho^{g} \vert_{H^{g} \cap H})}$ of $H$ depends only on the double coset $\bar g \in H \backslash G /  H$ up to a canonical isomorphism.

By {\it Mackey's formula} (see, e.g.,~\cite[Chapter~I, \S~5.5]{Vigneras}), there is a canonical isomorphism of representations of $H$
\begin{equation}\label{eq:mackey}
\ind ^{G}_{H} (\rho) \vert_{H} \simeq  \bigoplus_{\bar g \in  H \backslash G / H} \ind^{H}_{H^{g} \cap H} (\rho^{g} \vert_{H^{g} \cap H}) \, .
\end{equation}
Using the isomorphisms~\eqref{eq:Frob1} and~\eqref{eq:mackey}, we get a canonical isomorphism of vector spaces
\begin{equation}\label{eq:end}
\End_{G} \big( \ind^{G}_{H}(\rho) \big) \simeq \bigoplus_{ \bar g \in H \backslash G / H }  \Hom_{H} \big( \rho,  \ind^{H}_{H^{g} \cap H} (\rho^{g} \vert_{H^{g} \cap H} ) \big) \, .
\end{equation}

\begin{remark}\label{rmk:injend}
It follows from the isomorphism~\eqref{eq:mackey} that $\rho$ is canonically identified with a direct summand of the representation $\ind_H^G(\rho)\vert_H$. In particular, this implies that the natural homomorphism~${\End_H(\rho)\to \End_G\big(\ind_H^G(\rho)\big)}$ is injective.
\end{remark}

\medskip

\begin{lemma}\label{lemma:nofinite}
If the index of $H$ in $G$ is infinite, then the representation~${\ind_{H}^{G} (\rho)}$ of~$G$ does not have non-zero finite-dimensional subrepresentations.
\end{lemma}

\begin{proof}
Suppose that there is a non-zero finite-dimensional subrepresentation~$\tau$ of~$\ind_{H}^{G} (\rho)$. Let $X\subset H \backslash G$ be the union of the supports of all sections in the representation space of $\tau$ (see Definition~\ref{definition:compact}). Since $\tau$ is finite-dimensional and $\ind_H^G(\rho)$ is finitely induced, the set $X$ is finite. It can easily be checked that $X$ is invariant under the action of $G$ on $H \backslash G$ by right translations.

On the other hand, $G$ acts transitively on~$H \backslash G$, whence $X=H\backslash G$. By the assumption of the lemma, the set $H \backslash G$ is infinite, thus we get a contradiction.
\end{proof}

\medskip

Clearly, the subset $S(H)\subset G$ (see Definition~\ref{def:ss}) is invariant under left and right translations by elements of $H$. Combining the isomorphism~\eqref{eq:end} with Lemma~\ref{lemma:nofinite} and the isomorphism~\eqref{eq:Frob2}, we obtain the following fact.

\begin{proposition}\label{prop:end}
If $\rho$ is finite-dimensional, then there is a canonical isomorphism of vector spaces
$$
\End_{G} \big( \ind^{G}_{H}(\rho) \big) \simeq \bigoplus_{ \bar g \in H \backslash S(H) / H }  \Hom_{H^g \cap H} \big( \rho \vert_{H^g \cap H},  \rho^{g} \vert_{H^{g} \cap H}  \big) \, .
$$
\end{proposition}

Note that the vector space $\Hom_{H^g \cap H} \big( \rho \vert_{H^g \cap H},  \rho^{g} \vert_{H^{g} \cap H}  \big)$ depends only on the double coset $\bar g \in H \backslash G / H$ up to a canonical isomorphism.

\begin{remark}
An analog of Proposition~\ref{prop:end} for unitary representations was discovered by Mackey~\cite[Theorem $3'$]{Mac}. Note that for unitary representations, one replaces the set $S(H)$ by the subset $S(H)'\subset S(H)$ that consists of all elements $g\in G$ such that $H^g\cap H$ is of finite index in both $H$ and $H^g$. Example~\ref{examp:infinind} shows that $S(H)'\ne S(H)$ for an arbitrary group $G$. Nevertheless, Lemma~\ref{lemma:normalizer} and Proposition~\ref{prop:Brownnorm} imply the equality $S(H)'=S(H)$ when $G$ is a finitely generated nilpotent group.
\end{remark}

\medskip

Proposition~\ref{prop:end} motivates the following definition.

\begin{definition} \label{def:s}
\hspace{0cm}
\begin{itemize}
\item[(i)] Let $S(H, \rho) \subset G$ be the set of all elements $g \in S(H)$ such that
$$
\Hom_{H^g\cap H}\big(\rho \vert_{H^g \cap H}, \rho^{g} \vert_{H^g \cap H}\big)\ne 0\,.
$$
\item[(ii)] A pair $(H, \rho)$ is called \emph{perfect} if the subset $S(H, \rho)\subset G$ is a subgroup, the group~$H$ is normal in $S(H, \rho)$, and the index of $H$ in $S(H,\rho)$ is finite.
\end{itemize}
\end{definition}

Clearly, there is an embedding $H\subset S(H,\rho)$. Also, it is easily shown that the subset $S(H,\rho)\subset G$ is invariant under left and right translations by elements of $H$.

\begin{remark}\label{remark:perf}
\hspace{0cm}
\begin{itemize}
\item[(i)]
Suppose that for an element $g\in S(H,\rho)$, the representations $\rho \vert_{H^g \cap H}$ and $\rho^{g} \vert_{H^g \cap H}$ are irreducible. Since any non-zero morphism between irreducible representations is an isomorphism, this implies an isomorphism of representations $\rho \vert_{H^g \cap H}\simeq \rho^{g} \vert_{H^g \cap H}$. In particular, this holds in the following two cases: if $\rho=\chi$ is a character; if $\rho$ is irreducible, the subset $S(H, \rho)\subset G$ is a subgroup, and $H$ is normal in $S(H, \rho)$.
\item[(ii)]
Suppose that $\rho$ is irreducible and there is a subgroup $F\subset G$ such that~${S(H,\rho)}$ is contained in $F$ (in particular, we have $H\subset F$) and $H$ is normal in $F$. Then the group $F$ acts on $H$ by conjugation, which gives an action of $F$ on the set of isomorphism classes of representations of~$H$. It follows from item~(i) that $S(H,\rho)$ coincides with the stabilizer in $F$ of the isomorphism class of $\rho$ with respect to the latter action. Therefore the subset $S(H,\rho)\subset G$ is a subgroup and $H$ is normal in~${S(H,\rho)}$.
\end{itemize}
\end{remark}

Proposition~\ref{prop:end} implies directly the following fact.

\begin{corollary}\label{lemma:index}
If $\rho$ is finite-dimensional, then the following conditions are equivalent:
\begin{itemize}
\item[(i)]
the natural homomorphism $\End_H(\rho)\to\End_G\big(\ind_H^G(\rho)\big)$ is an isomorphism;
\item[(ii)]
there is an equality $S(H,\rho)=H$.
\end{itemize}
\end{corollary}

\begin{remark}
If $\rho=\chi$ is a character, then by Proposition~\ref{prop:end}, there is a canonical isomorphism of vector spaces
$$
\End_{G} \big( \ind_{H}^{G} (\chi) \big) \simeq \bigoplus_{H \backslash S(H, \chi) / H}K \, .
$$
\end{remark}

\subsection{Irreducible pairs}\label{subsect:Brown}

\begin{definition}\label{definition:weight}
\hspace{0cm}
\begin{itemize}
\item[(i)]
An {\it irreducible pair} is a pair $(H,\rho)$, where $H\subset G$ is a subgroup and~$\rho$ is a (non-zero) finite-dimensional irreducible representation of~$H$. A {\it weight pair} is a pair $(H,\chi)$, where $\chi$ is a character of $H$.
\item[(ii)]
Given an irreducible pair $(H,\rho)$, a finite-dimensional representation $\sigma$ of $H$ is {\it $\rho$-isotypic} if $\sigma\simeq \rho^{\oplus r}$ for some positive integer $r$.
\item[(iii)]
Define the following partial order on the set of irreducible pairs: put ${(H, \rho) \leqslant (H', \rho')}$ if and only if $H \subset H'$ and $\rho'|_H$ is $\rho$-isotypic.
\end{itemize}
\end{definition}

Given weight pairs $(H,\chi)$ and $(H',\chi')$, one has $(H,\chi)\leqslant (H',\chi')$ if and only if $H\subset H'$ and $\chi'|_H=\chi$.

\begin{lemma}\label{lemma:isot}
Let $(H,\rho)$ be an irreducible pair and $\sigma$ a subquotient of the representation $\rho^{\oplus r}$, where $r$ is a positive integer. Then $\sigma$ is~\mbox{$\rho$-iso\-ty\-pic}.
\end{lemma}
\begin{proof}
First suppose that $\sigma$ is an irreducible subrepresentation of~$\rho^{\oplus r}$. Looking at the projections~${\rho^{\oplus r}\to\rho}$ to each of $r$ natural direct summands of~$\rho^{\oplus r}$, we see that there is a non-zero projection $f\colon\sigma\to \rho$, say, to the $i$-th summand. The morphism $f$ is an isomorphism by the irreducibility of~$\sigma$ and~$\rho$. Furthermore, the subrepresentation~$\sigma$ splits out of~$\rho^{\oplus r}$. Indeed, the corresponding morphism $\rho^{\oplus r}\to\rho$ can be taken to be zero on all summands except for the $i$-th one and to be the inverse to~$f$ on the $i$-th summand.

Now let $\sigma\subset \rho^{\oplus r}$ be an arbitrary subrepresentation. Since $\sigma$ is finite-dimensional, there is an irreducible subrepresentation $\sigma'\subset \sigma$. By what was shown above, we see that $\sigma'\simeq\rho$ and $\sigma'$ is a direct summand of $\rho^{\oplus r}$. It follows that $\sigma'$ is a direct summand of $\sigma$ as well. Thus induction on dimension of $\sigma$ implies that $\sigma$ is $\rho$-isotypic.

By duality for finite-dimensional representations, we obtain that any quotient of $\rho^{\oplus r}$ is $\rho$-isotypic. This completes the proof.
\end{proof}

Recall that $\pi$ is a representation of $G$.

\begin{definition}\label{defin:pipair}
\hspace{0cm}
\begin{itemize}
\item[(i)]
A {\it $\pi$-irreducible pair} is an irreducible pair $(H,\rho)$ such that the vector space $\Hom_{H} (\rho, \pi \vert_{H} )$ is non-zero. A $\pi$-irreducible pair is {\it finite} if the vector space $\Hom_{H} (\rho, \pi \vert_{H} )$ is finite-dimensional. A {\it (finite) $\pi$-weight pair} is defined similarly.
\item[(ii)]
A representation $\pi$ \emph{has finite weight} if there is a finite $\pi$-weight pair.
\end{itemize}
\end{definition}

We will use the following simple observation.

\begin{remark}\label{remark:extend}
Let $(H,\rho)$ be a finite $\pi$-irreducible pair. Suppose that the subset $S(H,\rho)\subset G$ is a subgroup and $H$ is normal in $S(H,\rho)$. Let $W$ be the $\rho$-isotypic subspace of the representation space of $\pi$, that is, $W$ is the representation space of the image of the natural morphism of representations of $H$
$$
\rho\otimes_K\Hom_H(\rho,\pi|_H)\longrightarrow \pi|_H\,,
$$
where $H$ acts trivially on the vector space $\Hom_H(\rho,\pi|_H)$. Then $W$ is invariant under the action of $S(H,\rho)$. Also, by Lemma~\ref{lemma:isot}, the representation of $H$ on~$W$ is $\rho$-isotypic.
\end{remark}

\medskip

The following result allows us to extend $\pi$-irreducible pairs.

\begin{lemma}\label{lemma:extend}
Let $(H,\rho)$ be a $\pi$-irreducible pair and $g\in G$ an element such that $H^g=H$ and $\rho^g\simeq\rho$. Suppose that at least one of the following conditions holds:
\begin{itemize}
\item[(i)]
the $\pi$-irreducible pair $(H,\rho)$ is finite;
\item[(ii)]
there is a positive integer $n$ such that $g^n\in G$.
\end{itemize}
Then there is a $\pi$-irreducible pair $(H',\rho')$ such that~${(H,\rho)<(H',\rho')}$, where~${H'=\langle H,g \rangle}$.
\end{lemma}
\begin{proof}
Since any finite-dimensional representation contains an irreducible subrepresentation, by Lemma~\ref{lemma:isot}, it is enough to find a non-zero finite-dimensional subrepresentation of $\pi|_{H'}$ whose restriction to $H$ is $\rho$-isotypic.

If condition (i) holds, then Remark~\ref{remark:extend} provides the needed finite-dimensional subrepresentation of $\pi|_{H'}$, because $H'\subset S(H,\rho)$.

Suppose that condition (ii) holds true. Let $U_0$ be the representation space of the image of any non-zero morphism of representations $\rho\to\pi|_H$ and put
$$
U=\sum_{i=0}^{n-1}\pi(g^i)U_0\,.
$$
Clearly, $U$ is invariant under the action of the operator~$\pi(g)$. Since $H^g=H$, we see that $U$ is also is invariant under the action of the operators $\pi(h)$ for all $h\in H$. Finally, since $\rho^g\simeq \rho$, the representation of $H$ on~$U$ is a quotient of $\rho^{\oplus n}$. Hence by Lemma~\ref{lemma:isot}, the representation of $H$ on $U$ is $\rho$-isotypic. Thus~$U$ gives the needed finite-dimensional subrepresentation of~${\pi|_{H'}}$.
\end{proof}

\begin{example}\label{examp:extend}
Let $K=\CC$, let $G$ be the finite cyclic group $\ZZ/n\ZZ$, the element ${g\in G}$ its generator, the subgroup $H$ trivial, $\rho$ the trivial character of $H$, and $\pi$ the direct sum of a trivial infinite-dimensional representation of~$G$ with a non-trivial character $\psi$ of $G$. Then condition~(ii) of Lemma~\ref{lemma:extend} holds true. We have $H'=G$ and there are two possible options for the representation $\rho'$: the trivial character and the character $\psi$. Note that the vector space ${\Hom_{H'}(\rho',\pi)}$
is infinite-dimensional in the first case, while it is one-dimensional in the second case.
\end{example}

\begin{remark}\label{remark:mistake}
In particular, Example~\ref{examp:extend} shows that~\cite[Lemma~6]{Brown} is not correct (the mistake in the proof is that one uses an averaging operator which might vanish). \end{remark}

\section{Irreducibility of induced representations}\label{sect:irred}

\subsection{Irreducibility vs. Schur irreducibility}\label{sect:irr}

\begin{definition}\label{defin:scalar}
A representation $\pi$ of $G$ is called {\it Schur irreducible} if we have~${\End_{G} (\pi) = K}$.
\end{definition}

The following statement is an analog of classical Schur's lemma; for a proof see, e.g.,~\cite[Claim 2.11]{BZ} or~\cite[Chapter 5, \S~4.2]{Bernst}.

\begin{proposition}\label{prop:schur}
Suppose that the field $K$ is algebraically closed and uncountable. Then any countably dimensional irreducible representation over $K$ of an arbitrary group is Schur irreducible.
\end{proposition}

The following examples show that Proposition~\ref{prop:schur} is not valid for an arbitrary field $K$, even if one relaxes the condition $\End_G(\pi)=K$ to finite-dimensionality of $\End_G(\pi)$ over $K$.

\begin{example}
\hspace{0cm}
\begin{itemize}
\item[(i)]
Suppose that the field $K$ is algebraically closed and countable. By~${K(t)}$ denote the field of rational functions of $t$ over $K$.
Let $G$ be the group~${K(t)^*}$ of non-zero rational functions and let $\pi=K(t)$ with the action of~$G$ given by multiplication of rational functions. Then $\pi$ is countably dimensional and irreducible, while $\End_G(\pi)=K(t)\ne K$.
\item[(ii)]
Suppose that there is an extension of fields $K\subset L$ such that $L$ is infinitely countably dimensional as a $K$-vector space (the field $K$ might be uncountable).
Let $G=L^*$ and $\pi=L$ with the action of~$G$ given by multiplication of elements of $L$. Then $\pi$ is countably dimensional over~$K$ and irreducible, while $\End_G(\pi)=L\ne K$.
\end{itemize}
\end{example}

\begin{remark}\label{remark:schur}
It follows from Proposition~\ref{prop:schur} that for countable groups, irreducibility implies Schur irreducibility over an algebraically closed uncountable field.
\end{remark}

\medskip

In general, Schur irreducibility does not imply irreducibility as the following example shows.

\begin{example}\label{examp:contr}
Suppose that a proper subgroup $H\subset G$ satisfies $S(H)=H$ (see Example~\ref{exam:s}). In particular, the index of $H$ in $G$ is infinite. Suppose that the field $K$ is algebraically closed and
uncountable. Let $\tau$ be a finite-dimensional (irreducible) representation of $G$ such that $\tau|_H$ is irreducible. Consider the representation
${\pi=\ind_H^G(\tau\vert_H)}$ of $G$. By Corollary~\ref{lemma:index} and Proposition~\ref{prop:schur}, the representation~$\pi$ is Schur irreducible.

On the other hand, by the isomorphism~\eqref{eq:Frob1}, there is a non-zero morphism of representations from~$\pi$ to~$\tau$. This morphism is not an isomorphism, because the dimension
of $\pi$ is infinite and the dimension of $\tau$ is finite. Thus $\pi$ is not irreducible.
\end{example}

However the next result claims that a certain bound on endomorphisms still implies irreducibility for a wide range of representations. This fact is essential for our proof of the main theorem (see Subsection~\ref{subsect:proof}).

\begin{proposition}\label{lemma}
Suppose that $H$ is normal in $G$, a representation $\rho$ of~$H$ is irreducible, and the natural homomorphism $\End_H(\rho)\to \End_G\big(\ind_H^G(\rho)\big)$ is an isomorphism. Then~$\ind_H^G(\rho)$ is irreducible.
\end{proposition}
\begin{proof}
By $V$ denote the representation space of $\ind_{H}^{G} (\rho)$. Note that the representation $\ind_{H}^{G} (\rho)$ is irreducible if and only if any non-zero vector $v \in V$ generates~$V$ as a representation of $G$. Let us show that this condition holds true.

Since $H$ is normal in $G$, we have $H \backslash G / H = G / H$ and for any $g \in G$, there are equalities $H^g = H = H^g \cap H$. Therefore the isomorphisms~\eqref{eq:mackey},~\eqref{eq:end} take the forms
\begin{equation}\label{eq:mackey-norm}
\ind ^{G}_{H} (\rho) \vert_{H} \simeq  \bigoplus_{\bar g \in  G / H} \rho^{g} \,,
\end{equation}
\begin{equation}\label{hom}
\End_{G} \big( \ind^{G}_{H}(\rho) \big) \simeq \bigoplus_{ \bar g \in  G / H }  \Hom_{H} (\rho, \rho^{g} ) \, ,
\end{equation}
respectively. For every $\bar g\in G/ H$, by $V_{ \bar g}$ denote the representation space of $\rho^{g}$. In this notation, the isomorphism~\eqref{eq:mackey-norm} becomes
\begin{equation}\label{eq:mackey-normV}
V\simeq \bigoplus_{\bar g\in G/H} V_{\bar g}\,.
\end{equation}

Consider a non-zero vector $v \in V$. By the isomorphism~\eqref{eq:mackey-normV}, $v$ can be written as a sum
$$
v  = \sum_{\bar g \in G / H} v_{\bar g}\,,\qquad v_{\bar g} \in V_{\bar g}\,,
$$
where only finitely many summands are non-zero. Let $k$ be the number of the non-zero summands. Suppose that $k\geqslant 2$.

Let $\bar g\in G / H$ be such that $v_{\bar g}\ne 0$. By $I_{\bar g}$ denote the kernel of the action of the group algebra~$K[H]$ on the vector $v_{\bar g}$. Since $\rho$ is irreducible, the representation~$\rho^{ g}$ of~$H$ is irreducible as well, whence $v_{\bar g}$ generates $V_{\bar g}$ as a representation of $H$. Consequently we have an isomorphism of representations of $H$
$$
K[H] / I_{\bar g}\simeq \rho^{g}\,.
$$

Further, the isomorphism ${\End_H(\rho)\simeq \End_G\big(\ind^{G}_{H}(\rho)\big)}$ and the isomorphism~\eqref{hom} imply that the irreducible representations $\rho^{ g}$, $\bar g \in G / H$, are pairwise non-isomorphic. Therefore if $v_{\bar g_1}\ne 0$ and $v_{\bar g_2}\ne 0$, then the ideals $I_{\bar g_1}$ and $I_{\bar g_2}$ are different non-zero ideals. Permuting $g_1$ and $g_2$, if needed, we see that there is an element $P \in K[H]$ such that $P\in I_{\bar g_1}$ and $P\notin I_{\bar g_2}$, that is, $P(v_{\bar g_1}) = 0$ and $P(v_{\bar g_2}) \ne 0$ (actually, none of the ideals~$I_{\bar g_1}$ and~$I_{\bar g_2}$ contains another one, because the representations~$\rho^{g_1}$ and~$\rho^{g_2}$ are irreducible). By construction, the vector
$$
P(v)=\sum_{\bar g\in H\backslash G}P(v_{\bar g})\,,\qquad P(v_{\bar g})\in V_{\bar g}\,,
$$
is non-zero and has a strictly less number of non-zero summands with respect to the decomposition~\eqref{eq:mackey-normV}.

Thus we may suppose that $k = 1$, that is, $v = v_{\bar g}\ne 0$ for some $\bar g\in G / H$. As explained above, the vector $v_{\bar g}$ generates $V_{\bar g}$ as a representation of $H$. Moreover, the action of an element $g' \in G$ sends $v_{\bar g}$ to a non-zero vector $v_{\bar g \bar g'}\in V_{\bar g \bar g'}$, which, in turn, generates $V_{\bar g \bar g'}$ as a representation of $H$. It follows that~$v_{\bar g}$ generates~$V$ as a representation of $G$, which completes the proof of the proposition.
\end{proof}

A particular case of Proposition~\ref{lemma} was proved by Arnal and Parshin~\cite[Theorem 2]{Parshin-Arnal}.

\begin{remark}\label{remark:equivSchur}
The assumptions of Proposition~\ref{lemma} are equivalent to irreducibility of $\rho$ and Schur irreducibility of $\ind_H^G(\rho)$ in the following two cases: $\rho=\chi$ is a character; the group $G$ is countable and the field $K$ is algebraically closed and uncountable (see Remark~\ref{remark:schur}). Moreover, if $K$ is algebraically closed and uncountable, then the converse to the implication of Proposition~\ref{lemma} holds true.
\end{remark}

Example~\ref{examp:contr} shows that Proposition~\ref{lemma} does not hold when the subgroup~${H\subset G}$ is not necessarily normal. Further, the next example shows that the converse to the implication of Proposition~\ref{lemma} does not hold over an arbitrary field $K$.

\begin{example}
Let $K=\QQ(i)$, $G=\ZZ/8\ZZ$, $H=\ZZ/4\ZZ$, and $\rho=\chi$ be a primitive character of $\ZZ/4\ZZ$ over $K$. The two-dimensional representation~${\ind_H^G(\chi)}$ is irreducible, because the character $\chi$ does not extend to a character of~$G$ over $K$. On the other hand, we have $\End_G\big(\ind_H^G(\chi)\big)\simeq K(\zeta)$, where~$\zeta$ is a primitive root of unity of degree $8$, thus the natural homomorphism ${\End_H(\chi)\to \End_G\big(\ind_H^G(\chi)\big)}$ is not an isomorphism.
\end{example}

\medskip

Proposition~\ref{lemma} implies the following general result.

\begin{corollary} \label{theorem:irr}
Suppose that there exists a sequence of subgroups
$$
G = G_0 \supset G_{1} \supset \ldots \supset G_{n-1} \supset G_n = H \,,
$$
such that $G_{i}$ is normal in $G_{i-1}$ for any $i$, $1 \leqslant i \leqslant n$. Suppose that a representation $\rho$ of $H$ is irreducible and the natural homomorphism ${\End_H(\rho)\to \End_G\big(\ind_H^G(\rho)\big)}$ is an isomorphism. Then $\ind_{H}^{G} (\rho)$ is irreducible.
\end{corollary}
\begin{proof}
The proof is by induction on $n$. Combining the isomorphism of representations
$$
\ind_{H}^{G}(\rho) \simeq  \ind^{G}_{G_{n-1}} \big( \ind_{H}^{G_{n-1}}(\rho) \big)
$$
with Remark~\ref{rmk:injend}, we see that the natural homomorphism ${\End_H(\rho)\to \End_{G_{n-1}}\big(\ind_H^{G_{n-1}}(\rho)\big)}$ is an isomorphism. Therefore, by Proposition~\ref{lemma}, the representation $\ind_{H}^{G_{n-1}} (\rho)$ is irreducible. We conclude by the induction hypothesis applied to the subgroup~${G_{n-1}\subset G}$.
\end{proof}

\subsection{Induced representations of nilpotent groups}\label{sect:nilp}

Suppose that $G$ is a nilpotent group, that is, its lower central series is finite:
$$ G=
\gamma_0(G) \supset \gamma_{1}(G) \supset \ldots \supset \gamma_{n-1}(G) \supset \gamma_n(G) = \{ e \} \,  .$$

\begin{lemma}\label{lemma:filter_nilp}
For any subgroup $H\subset G$, there exist a sequence of subgroups
$$
G = G_0 \supset G_{1} \supset \ldots \supset G_{n-1} \supset G_n = H \,
$$
such that $G_{i}$ is normal in $G_{i-1}$ for any $i$, $1 \leqslant i \leqslant n$.
\end{lemma}

\begin{proof}
Let $G_i = \langle H, \gamma_{i}(G) \rangle$ be the subgroup of $G$ generated by $H$ and $\gamma_{i}(G)$, $0 \leqslant i \leqslant n$. In order to prove that $G_i$ is normal in $G_{i-1}$, it is enough to show that $[G_{i-1}, G_{i}] \subset G_{i}$.
This follows from the embeddings
$$
[H, H] \subset H\subset G_i\, ,
$$
$$
[\gamma_{i-1}(G),  H] \subset [\gamma_{i-1}(G),  G] =  \gamma_{i}(G) \subset G_{i}\,,
$$
$$
[\gamma_{i-1}(G),  \gamma_{i}(G)] \subset [G, \gamma_{i}(G)] =  \gamma_{i+1}(G) \subset G_{i}\, ,
$$
$$
[H, \gamma_{i}(G)]  \subset [G, \gamma_{i}(G)] =  \gamma_{i+1}(G) \subset G_{i}\, .
$$
\end{proof}

\medskip

Combining Corollaries~\ref{lemma:index} and~\ref{theorem:irr} with Lemma~\ref{lemma:filter_nilp}, we obtain the following useful result.

\begin{theorem}\label{theor:char}
Let $G$ be a nilpotent group and $(H,\rho)$ be an irreducible pair (see Definition~\ref{definition:weight}(i)). Suppose that ${S(H, \rho) = H}$ (see Definition~\ref{def:s}(i)). Then the representation $\ind^G_H (\rho)$ of $G$ is irreducible.
\end{theorem}

Recall that if the field $K$ is algebraically closed, then any finite-dimensional irreducible representation over $K$ is Schur irreducible. Therefore in this case, Theorem~\ref{theor:char} claims the following: Schur irreducibility implies irreducibility for representations of type $\ind_H^G(\rho)$, where $(H,\rho)$ is an irreducible pair in a finitely generated nilpotent group (if, in addition, $K$ is uncountable, then the converse implication holds as well).

In the next subsection, we show that Schur irreducibility does not imply irreducibility for arbitrary representations of finitely generated nilpotent groups (even if representations are over an algebraically closed uncountable field).

\subsection{Example: the Heisenberg group}\label{subsect:Heis}

Recall that the {\it Heisenberg group} over a commutative unital ring is the group of $3\times 3$ upper triangular matrices with units on the diagonal and with coefficients in the ring. Put
$$
x =
\begin{pmatrix}
     1 & 1 & 0 \\
     0 & 1 & 0 \\
     0 & 0 & 1
\end{pmatrix}
, \quad
y =
\begin{pmatrix}
     1 & 0 & 0 \\
     0 & 1 & 1 \\
     0 & 0 & 1 \\
\end{pmatrix}
, \quad
z =
\begin{pmatrix}
     1 & 0 & 1 \\
     0 & 1 & 0 \\
     0 & 0 & 1 \\
\end{pmatrix}.
$$
We have the relation $xy=z\,yx$.

\medskip

Below we consider the Heisenberg group $G$ over the ring of integers. Fix a non-zero element $c\in K$. It turns out that representations of $G$ such that $z$ acts by $c$ admit the following geometric description.

\medskip

Let us denote by $R$ the $K$-algebra of Laurent polynomials~${K[t,t^{-1}]}$. The $K$-variety~${\GG_m=\Spec(R)}$ is the one-dimensional algebraic torus over $K$. Let $\gamma\colon R\to R$ be the automorphism of the $K$-algebra $R$ such that~${\gamma(t)=c\,t}$. Equivalently, $\gamma$ is the automorphism of $\GG_m$ given by the group translation by the element $c\in\GG_m$. Let $\Gamma$ be the cyclic abelian group generated by the automorphism $\gamma$. By construction, the group $\Gamma$ acts on the $K$-algebra~$R$ and on the algebraic variety $\GG_m$.

A {\it $\Gamma$-equivariant $R$-module} is an $R$-module $M$ together with a $K$-linear action of $\Gamma$ on $M$ such that $\gamma(fm)=\gamma(f)\gamma(m)$ for all elements $f\in R$, $m\in M$. Morphisms between $\Gamma$-equivariant $R$-modules are defined naturally. For instance, clearly, $R$ has a canonical structure of a $\Gamma$-equivariant $R$-module.

In geometric terms, a $\Gamma$-equivariant $R$-module is the same as a~$\Gamma$-equivariant quasi-coherent sheaf on $\GG_m$. In particular,~$R$ as a $\Gamma$-equivariant $R$-module corresponds to the structure sheaf of $\GG_m$ with its canonical $\Gamma$-equivariant structure.

\medskip

Let $\pi$ be a representation of $G$ such that $\pi(z)=c$ and $M$ the representation space of $\pi$. Consider an $R$-module structure on $M$ such that $t$ acts by the operator $\pi(y)$. Let $\gamma$ act on~$M$ by the operator $\pi(x)$. Then $M$ becomes a~$\Gamma$-equivariant \mbox{$R$-module}, because of the relation~${\pi(x)\pi(y)=c\,\pi(y)\pi(x)}$.

One checks easily that the assignment $\pi\longmapsto M$ defines an equivalence (actually, an isomorphism) between the category of representations of $G$ such that~$z$ acts by $c$ and the category of $\Gamma$-equivariant $R$-modules.

\medskip

Now suppose that the non-zero element $c\in K$ is not a root of unity. Let $P$ be the $R$-module that consists of all rational functions on $\GG_m$ that have poles of order at most one at the points $c^i\in\GG_m$, $i\in\ZZ$, and are regular elsewhere. Define also the $R$-module
$$
Q=\bigoplus_{i\in \ZZ}R/(t-c^i)\,.
$$
The corresponding quasi-coherent sheaf on $\GG_m$ is the direct sum of the skyscraper sheaves at the points $c^i$, $i\in \ZZ$.

The action of $\Gamma$ on $\GG_m$ leads to natural $\Gamma$-equivariant structures on~$P$ and~$Q$. Moreover, we have an exact sequence of $\Gamma$-equivariant $R$-modules
\begin{equation}\label{eq:extension}
0\longrightarrow R\longrightarrow P\longrightarrow Q\longrightarrow 0\,.
\end{equation}
The exact sequence~\eqref{eq:extension} does not split, because the $R$-module $Q$ is a torsion module and $R$ is torsion-free.

Let us show that $R$ and $Q$ are irreducible $\Gamma$-equivariant $R$-modules. Let~${I\subset R}$ be a $\Gamma$-equivariant submodule. Then $I$ is an ideal in $R$, being an $R$-submodule. On the other hand, for any $\Gamma$-equivariant module, its support on $\GG_m$ is invariant under the action of~$\Gamma$. Applying this to the $\Gamma$-equivariant module~$R/I$ and using that $c$ is not a root of unity, we obtain that either $I=0$, or $I=R$, whence $R$ is irreducible. Irreducibility of $Q$ is proved similarly.

Further, the $\Gamma$-equivariant $R$-modules $R$ and $Q$ are not isomorphic, being non-isomorphic \mbox{$R$-modules}. We see that $P$ is a non-trivial extension between two non-isomorphic irreducible $\Gamma$-equivariant $R$-modules $Q$ and $R$. In particular, $P$ is not irreducible.

\medskip

Let us prove that $P$ is Schur irreducible as a $\Gamma$-equivariant $R$-module. First we show that $R$ is Schur irreducible as a $\Gamma$-equivariant $R$-module. Indeed, the ring of endomorphisms of $R$ as an $R$-module is isomorphic to~$R$. Further, the ring of endomorphisms of $R$ that respect the $\Gamma$-equivariant structure is identified with the $\Gamma$-invariant part of $R$. Since $c$ is not a root of unity, the $\Gamma$-invariant part of $R$ is just~$K$.

Now let $\varphi\colon P\to P$ be an endomorphism of $P$ as a $\Gamma$-equivariant $R$-module. The composition
$$
R\stackrel{\varphi|_R}\longrightarrow P\longrightarrow Q
$$
is equal to zero, because $R$ and $Q$ are non-isomorphic irreducible $\Gamma$-equivariant $R$-modules. Therefore $R$ is invariant under the action of $\varphi$. By Schur irreducibility of $R$, we see that $\varphi|_R=\lambda$ for an element $\lambda\in K$. The morphism $\varphi-\lambda\colon P\to P$ vanishes on $R$, whence it factors through $Q$. Since the exact sequence~\eqref{eq:extension} does not split, we see that $\varphi-\lambda=0$, that is, $\varphi=\lambda$ and $P$ is Schur irreducible.

\medskip

Using the above equivalence of categories, we see that Schur irreducibility does not imply irreducibility for (possibly, complex) representations of the Heisenberg group over~$\ZZ$.

\section{Main results}\label{sect:main}

\subsection{Monomial and finite weight representations}\label{subsect:stat}

Recall that a representation $\pi$ of $G$ is \emph{monomial} if there is a weight pair $(H,\chi)$ (see Definition~\ref{definition:weight}(i)) such that $\pi \simeq \ind_H^G (\chi)$.

\begin{proposition}\label{prop:weight}
Suppose that the group $G$ is countable and the field $K$ is algebraically closed and uncountable. Let $\pi$ be an irreducible representation of~$G$ over~$K$.
Then the following holds true:
\begin{itemize}
\item[(i)]
if $\pi$ is isomorphic to a finitely induced representation $\ind_H^G(\rho)$, where $H\subset G$ is a subgroup and $\rho$ is a representation of $H$, then the vector space~$\Hom_H(\rho,\pi|_H)$ is one-dimensional;
\item[(ii)]
if $\pi$ is monomial, then $\pi$ has finite weight (see Definition~\ref{defin:pipair}(ii)).
\end{itemize}
\end{proposition}
\begin{proof}
Item~(i) follows from the isomorphism~\eqref{eq:Frob1} and Remark~\ref{remark:schur}. Item~(ii) follows directly from item~(i).
\end{proof}

\medskip

Here is our key result.

\begin{theorem} \label{theorem:key}
Let $G$ be a finitely generated nilpotent group and $\pi$ an irreducible representation of $G$ over an arbitrary field $K$ such that there is a finite $\pi$-irreducible pair (see Definition~\ref{defin:pipair}(i)). Then there is an irreducible pair~${(H,\rho)}$ (see Definition~\ref{definition:weight}(i)) such that $\pi\simeq\ind_H^G(\rho)$.
\end{theorem}

The proof of Theorem~\ref{theorem:key} is given in Subsection~\ref{subsect:proof}. It consists in an explicit construction of a pair~${(H,\rho)}$ such that~${\pi\simeq\ind_{H}^G(\rho)}$. The construction goes as follows (we refer to steps in Subsection~\ref{subsect:proof}). We start with a maximal finite $\pi$-irreducible pair (see Step 1). Then we replace it by a certain finite index subgroup in order to get a perfect $\pi$-irreducible pair~${(H_0,\rho_0)}$ (see Step~2). Notice that $(H_0,\rho_0)$ is not necessarily finite. Now the existence of a perfect $\pi$-irreducible pair allows us to take a maximal perfect $\pi$-irreducible pair~${(H,\rho)}$ with respect to the order from Definition~\ref{definition:weight}(iii). We prove the equality $S(H,\rho)=H$ (see Step 3). Finally, Theorem~\ref{theor:char} implies that the representation $\ind_H^G(\rho)$ is irreducible and Frobenius reciprocity gives a non-zero morphism of irreducible representations $\ind_H^G(\rho)\to \pi$, which is necessarily an isomorphism (see Step~4).

\medskip

The following result is well-known and its proof essentially repeats that of~\cite[\S 8.5, Theorem 16]{Serre} (cf.~\cite[Lemma 1]{Brown}). We provide the proof for convenience of the reader.

\begin{proposition}\label{prop:indfin}
Let $G$ be a finitely generated nilpotent group and $\pi$ an irreducible representation of $G$ over an algebraically closed field $K$ such that $\pi$ is finite-dimensional. Then $\pi$ is monomial.
\end{proposition}
\begin{proof}
The proof is by induction on the dimension of $\pi$. We can assume that the representation $\pi$ is faithful. There is an abelian normal subgroup $E\subset G$ that is not contained in the center of $G$. Indeed,~$E$ can be taken to be generated by the center of~$G$ and any non-central element in the previous term of the lower central series of~$G$.

Since $K$ is algebraically closed, there is a character $\chi$ of $E$ such that the vector space $\Hom_E(\chi,\pi|_E)$ is non-zero. Thus $(E,\chi)$ is a finite $\pi$-weight pair (see Definition~\ref{defin:pipair}(i)). Let $W$ be the $\chi$-isotypic subspace of the representation space of $\pi$ (see Remark~\ref{remark:extend}), that is, $W$ consists of all vectors in the representation space of $\pi$ on which the group $E$ acts by the character $\chi$.

Combining Remarks~\ref{remark:perf}(ii) and~\ref{remark:extend}, we obtain that the subset~${S(E,\chi)\subset G}$ is a subgroup and $W$ is invariant under the action of $S(E,\chi)$. Put $H=S(E,\chi)$ and let $\rho$ be a (non-zero) irreducible subrepresentation of the representation of $H$ on the finite-dimensional vector space $W$ (in particular, $\rho$ is a subrepresentation of $\pi|_H$). Clearly, the representation $\rho|_E$ of $E$ is $\chi$-isotypic (cf. Lemma~\ref{lemma:isot}).

One checks directly that there is an embedding $S(H,\rho)\subset S(E,\chi)$, whence we have~${S(H,\rho)=H}$. By Theorem~\ref{theor:char}, the representation~${\ind_H^G(\rho)}$ is irreducible. By the isomorphism~\eqref{eq:Frob1}, we have a non-zero morphism of representations~${\ind_H^G(\rho)\to \pi}$. Thus irreducibility of $\pi$ implies that this is an isomorphism.

Since $\pi$ faithful and $E$ is not contained in the center of $G$, we see that $\pi|_E$ is not $\chi$-isotypic, whence $\rho\ne \pi|_{H}$ and the dimension of $\rho$ is strictly less than the dimension of $\pi$. We conclude by the inductive hypothesis applied to the representation $\rho$ of $H$.
\end{proof}

Combining Theorem~\ref{theorem:key} with Proposition~\ref{prop:indfin}, we obtain the main result of the paper.

\begin{theorem} \label{theorem:main}
Let $G$ be a finitely generated nilpotent group and $\pi$ an irreducible representation of $G$ over an algebraically closed field $K$ such that $\pi$ has finite weight. Then $\pi$ is monomial.
\end{theorem}
\begin{proof}
By Theorem~\ref{theorem:key}, there is an irreducible pair $(H,\rho)$ such that~${\pi\simeq\ind_H^G(\rho)}$. Since $\rho$ is finite-dimensional, by Proposition~\ref{prop:indfin}, there is a weight pair $(H',\chi)$, where $H'\subset H$, such that $\rho\simeq\ind_{H'}^H(\chi)$. Therefore $\pi\simeq\ind_{H'}^G(\chi)$, which proves the theorem.
\end{proof}

\subsection{Proof of Theorem~\ref{theorem:key}}\label{subsect:proof}

The proof proceeds in several steps.

\medskip

{\it Step 1.}

Recall that the group~$G$ is Noetherian, being finitely generated and nilpotent~\cite[Theorem~2.18]{Man}. Hence there is a maximal finite \mbox{$\pi$-irreducible} pair, that is, a finite \mbox{$\pi$-irreducible} pair $(H,\rho)$ such that~${(H,\rho)}$ is maximal among all finite $\pi$-irreducible pairs with respect to the order on irreducible pairs (see Definition~\ref{definition:weight}(iii)).

\medskip

{\it Step 2.}

Let us prove that there exists a perfect $\pi$-irreducible pair (see Definition~\ref{def:s}(ii)). Let $(H,\rho)$ be a maximal finite $\pi$-irreducible pair, which exists by Step 1. Put (see Definition~\ref{def:ss} for $S(H)$)
\begin{equation}\label{eq:int}
H_0=\bigcap\limits_{{g \in S(H)}} H^g\, .
\end{equation}
Let $\rho_0$ be a (non-zero) irreducible subrepresentation of $\rho|_{H_0}$ (recall that $\rho$ is finite-dimensional). Clearly, $(H_0,\rho_0)$ is a $\pi$-irreducible pair as $(H,\rho)$ is so. Notice that we do not claim that~${(H_0,\rho_0)}$ is a finite $\pi$-irreducible pair (cf. Example~\ref{examp:extend}).

Let us show that the pair $(H_0,\rho_0)$ is perfect. By Theorem~\ref{theor:S}(i),(ii), the subset $S(H)\subset G$ is a subgroup and the index of $N_G(H)$ in $S(H)$ is finite. Therefore the intersection in the formula~\eqref{eq:int} is taken over a finite number of subgroups $H^g\cap H$ of finite index in $H$, whence the index of~$H_0$ in~$H$ is finite as well. Also, by construction, the group $H_0$ is normal in $S(H)$. Since the index of $H_0$ in $H$ is finite, by Theorem~\ref{theor:S}(iii), we have the equality ${S(H_0)=S(H)}$. Thus we have the embeddings of groups
$$
H_0\subset H\subset N_G(H)\subset S(H)=S(H_0)\,.
$$

By Remark~\ref{remark:perf}(ii) applied to $F=S(H_0)$, we see that the subset~${S(H_0,\rho_0)\subset G}$ is a subgroup and $H_0$ is normal in $S(H_0,\rho_0)$. It remains to prove that the index of $H_0$ in $S(H_0,\rho_0)$ is finite. Assume the converse. By Proposition~\ref{lemma:finite-index}, $S(H_0,\rho_0)$ is not contained in $H_0^*$, that is, there is an element $g\in S(H_0,\rho_0)$ such that $g^i\notin H_0$ for any positive integer~$i$. In particular, $g^i\notin H$ for any positive integer~$i$, because the index of $H_0$ in $H$ is finite.

Again by Theorem~\ref{theor:S}(ii), the index of $N_G(H)\cap S(H_0,\rho_0)$ in $S(H_0,\rho_0)$ is finite, because $S(H_0,\rho_0)$ is a subgroup of $S(H_0)=S(H)$. Therefore, changing~$g$ by its positive power, we may assume that $H^g=H$.

Let $C$ be the infinite cyclic group generated by $g$. Then $C$ acts on $H$ by conjugation, which gives the action of $C$ on the set of isomorphism classes of irreducible representations of $H$. We claim that the $C$-orbit of the isomorphism class of $\rho$ is finite. Indeed, let~$\Upsilon$ be the set of isomorphism classes of irreducible representations of $H$ that are quotients of the representation $\ind_{H_0}^H(\rho_0)$. The embedding~${C\subset S(H_0,\rho_0)}$ implies that the set $\Upsilon$ is invariant under the latter action of $C$. Since the index of $H_0$ in $H$ is finite, the representation $\ind_{H_0}^H(\rho_0)$ is finite-dimensional. This implies that the set $\Upsilon$ is finite. Finally, it follows from the isomorphism~\eqref{eq:Frob1} that the isomorphism class of $\rho$ belongs to $\Upsilon$. Thus the $C$-orbit of the isomorphism class of $\rho$ is finite, being contained in~$\Upsilon$. Therefore, changing $g$ by its positive power, we may assume further that $\rho^g=\rho$.

Since $(H,\rho)$ is a finite $\pi$-irreducible pair, condition (i) of Lemma~\ref{lemma:extend} is satisfied. Applying this lemma, we see that there is a $\pi$-irreducible pair such that $(H,\rho)<(H',\rho')$, where $H'=\langle H,g\rangle$. Since $\rho'|_{H}\simeq \rho^{\oplus r}$ for some positive integer $r$, we see that
$$
\Hom_{H'}(\rho',\pi|_{H'})\subset \Hom_{H}(\rho^{\oplus r},\pi|_H)\simeq \Hom_H(\rho,\pi|_H)^{\oplus r}\,.
$$
Consequently the $\pi$-irreducible pair $(H',\rho')$ is finite, which contradicts maximality of the finite $\pi$-irreducible pair $(H,\rho)$.

\medskip

{\it Step 3.}

Combining Step 2 with the fact that the group $G$ is Noetherian (cf. Step 1), we see that there is a maximal perfect \mbox{$\pi$-irreducible} pair $(H, \rho)$. Let us prove that~${S(H, \rho) = H}$. Assume the converse.

Since $(H,\rho)$ is perfect, we have a well-defined quotient group $S(H,\rho)/H$, which is finite and nilpotent. Therefore there is an element $z \in S(H, \rho)$ such that $z\notin H$ and the image of $z$ in $S(H, \rho) / H$ belongs to the center of $S(H,\rho)/H$. Condition (ii) of Lemma~\ref{lemma:extend} is satisfied for $z$. Applying this lemma, we obtain a $\pi$-irreducible pair $(H',\rho')$ with $H'=\langle H,z\rangle$ such that~${(H,\rho) < (H', \rho')}$.

Let us show that the pair $(H',\rho')$ is perfect. For this purpose, we first prove that~${S(H',\rho')}$ is contained in $S(H,\rho)$. Consider an element $g\in S(H',\rho')$, that is, $g\in S(H')$ and there is a non-zero morphism  ${\rho'|_{(H')^g\cap H'}\to(\rho')^g|_{(H')^g\cap H'}}$. Since $(H,\rho)<(H',\rho')$, we have that $\rho'|_H\simeq \rho^{\oplus r}$ for some positive integer~$r$. Hence there are isomorphisms of representations
$$
\rho'|_{H^g\cap H}\simeq (\rho|_{H^g\cap H})^{\oplus r}\,,\quad (\rho')^g|_{H^g\cap H}\simeq (\rho^g|_{H^g\cap H})^{\oplus r}\,.
$$
Clearly, $H^g\cap H$ is a subgroup of $(H')^g\cap H'$. This implies the embedding
$$
\Hom_{(H')^g\cap H'}\big({\rho'|_{(H')^g\cap H'},(\rho')^g|_{(H')^g\cap H'}}\big)\subset
\Hom_{H^g\cap H}\big({\rho|_{H^g\cap H},\rho^g|_{H^g\cap H}}\big)^{\oplus r^2}\,.
$$
Additionally, since the index of $H$ in~$H'$ is finite, by Theorem~\ref{theor:S}(iii), we have the equality $S(H)=S(H')$. Hence the index of $H^g\cap H$ in $H$ is finite. All together this implies that $g\in S(H,\rho)$, thus we have the embedding~${S(H',\rho')\subset S(H,\rho)}$.

Furthermore, since the image of $z$ in the quotient group~${S(H, \rho) / H}$ belongs to the center, $H'$ is normal in $S(H, \rho)$. Thus by Remark~\ref{remark:perf}(ii) applied to~${F=S(H,\rho)}$, the subset $S(H',\rho')\subset G$ is a subgroup and $H'$ is normal in~${S(H',\rho')}$.

Finally, the index of $H'$ in $S(H',\rho')$ is finite, because we have the embeddings of groups
$$
H\subset H'\subset S(H',\rho')\subset S(H,\rho)
$$
and the index of $H$ in $S(H,\rho)$ is finite as $(H,\rho)$ is perfect. We have shown that the~$\pi$-irreducible pair~$(H', \rho')$ is perfect, which contradicts maximality of the perfect $\pi$-irreducible pair $(H, \rho)$.

\medskip

{\it Step 4.}

As in Step 3, let $(H,\rho)$ be a maximal perfect $\pi$-irreducible pair. Since by Step 3 there is an equality $S(H,\rho)=H$, Theorem~\ref{theor:char} implies that the representation~${\ind_H^G (\rho)}$ is irreducible.

On the other hand, since $(H,\rho)$ is a $\pi$-irreducible pair, the isomorphism~\eqref{eq:Frob1} implies that there is a non-zero morphism of representations from $\ind^G_H (\rho)$ to $\pi$. Since the representations $\ind_H^G(\rho)$ and $\pi$ are irreducible, this is an isomorphism, which proves Theorem~\ref{theorem:key}.

\medskip

\begin{remark}\label{remark:finite}
Suppose that the field $K$ is algebraically closed and uncountable. Then the $\pi$-irreducible pair $(H,\rho)$ from Step 4 of the proof of Theorem~\ref{theorem:key} is finite by Proposition~\ref{prop:weight}(i).
\end{remark}

Recall that a {\it torsion-free rank} of a finitely generated nilpotent group $G$ is the sum of the ranks of the adjoint quotients of the lower central series (see, e.g.,~\cite[Chapter~0]{Baumslag}). One shows easily that the index of a subgroup $H$ of $G$ is finite if and only if $G$ and $H$ have the same torsion-free ranks.

\begin{remark}\label{remark:maxrank}
Suppose that the field $K$ is algebraically closed and uncountable. Let $(H',\rho')$ be a maximal finite $\pi$-irreducible pair such that the torsion-free rank of~$H'$ is also maximal. It follows from the proof of Theorem~\ref{theorem:key} and Remark~\ref{remark:finite} that there exists a finite $\pi$-irreducible pair $(H,\rho)$ such that $\pi\simeq \ind_H^G(\rho)$ and there is a finite index subgroup $H_0$ in both $H$ and $H'$. Equivalently, we have the equality~${H^*=(H')^*}$.
\end{remark}

\subsection{Isomorphic finitely induced representations}\label{subsect:isommonom}

Let~$G$ be an arbitrary group, let~$H_1$ and~$H_2$ be subgroups of~$G$, and let~$\rho_1$ and~$\rho_2$ be representations of~$H_1$ and~$H_2$, respectively. Let $S(H_1,H_2)$ be the set of all elements $g\in G$ such that the index of $H_2^g\cap H_1$ in $H_1$ is finite.

\begin{lemma}\label{lemma:twogroups}
Suppose that $\rho_1$ is finite-dimensional and the representations $\ind_{H_i}^G(\rho_i)$, $i=1,2$, are irreducible. Then the following conditions are equivalent:
\begin{itemize}
\item[(i)]
there is an isomorphism of representations $\ind_{H_1}^G(\rho_1)\simeq \ind_{H_2}^G(\rho_2)$;
\item[(ii)]
there exists an element $g\in S(H_1,H_2)$ such that there is a non-zero morphism of representations~${\rho_1|_{H_2^{g}\cap H_1}\to \rho_2^{g}|_{H_2^{g}\cap H_1}}$;
\item[(iii)]
there exists an element $g'\in S(H_2,H_1)$ such that there is a non-zero morphism of representations~$\rho_2|_{H_1^{g'}\cap H_2}\to \rho_1^{g'}|_{H_1^{g'}\cap H_2}$.
\end{itemize}
\end{lemma}
\begin{proof}
A similar argument as in the proof of Proposition~\ref{prop:end} together with a more general form of Mackey's isomorphism~\eqref{eq:mackey} (see, e.g.,~\cite[Chapter I, \S 5.5]{Vigneras}) implies the following canonical isomorphism of vector spaces:
$$
\Hom_{G}\big( \ind^{G}_{H_1}(\rho_1),\ind_{H_2}^G(\rho_2)\big) \simeq \bigoplus_{ \bar g \in H_2 \backslash S(H_1,H_2) / H_1 }  \Hom_{H_2^g \cap H_1} \big( \rho_1 \vert_{H_2^g \cap H_1},  \rho_2^{g} \vert_{H_2^{g} \cap H_1}  \big)\,.
$$
This proves the lemma.
\end{proof}

\begin{lemma}\label{lemma:12}
Suppose that $G$ is a finitely generated nilpotent group and the set~${S(H_2,H_1)}$ is non-empty. Then $S(H_1,H_2)$ coincides with the set of all elements $g\in G$ such that $(H_2^*)^g=H_1^*$.
\end{lemma}
\begin{proof}
Consider an element $g \in S(H_1,H_2)$. A similar argument as in the proof of Lemma~\ref{lemma:normalizer} shows that there is an embedding $H_1^*\subset (H_2^*)^g$. Similarly, for any $g'\in S(H_2,H_1)$, we get the embedding $H_2^*\subset (H_1^*)^{g'}$, whence we have
$$
H_1^*\subset (H_2^*)^g\subset (H_1^*)^{gg'}\,.
$$
Since $G$ is Noetherian, these embeddings are, in fact, equalities. Thus we have the equality~${(H_2^*)^g=H_1^*}$.

Now suppose that $(H_2^*)^g=H_1^*$. Using Remarks~\ref{remark:trivial}(i) and~\ref{remark:intersect}, we obtain the equality~${(H_2^g\cap H_1)^*=H_1^*}$. By Proposition~\ref{lemma:finite-index}, the index of ${H_2^g\cap H_1}$ in ${(H_2^g\cap H_1)^*=H_1^*}$ is finite. Therefore the index of $H_2^g\cap H_1$ in $H_1$ is finite as well, that is,~${g\in S(H_1,H_2)}$.
\end{proof}

\medskip

Lemmas~\ref{lemma:twogroups} and~\ref{lemma:12} imply the following criterion of isomorphism between finitely induced representations (cf.~\cite[Theorem 2]{Brown}).

\begin{proposition}\label{prop:twosubgroups}
Let $G$ be a finitely generated nilpotent group and let~${(H_1,\rho_1)}$ and~${(H_2,\rho_2)}$ be two irreducible pairs. Suppose that the representations $\ind_{H_1}^G(\rho_1)$ and $\ind_{H_2}^G(\rho_2)$ of $G$ are irreducible. Then there is an isomorphism of representations~${\ind_{H_1}^G(\rho_1)\simeq \ind_{H_2}^G(\rho_2)}$ if and only if there exists $g\in G$ such that~${(H_2^*)^g=H_1^*}$ and there is a non-zero morphism of representations~${\rho_1|_{H_2^g\cap H_1}\to\rho_2^g|_{H_2^g\cap H_1}}$.
\end{proposition}

\medskip

The following example shows that, in general, one can not strengthen Proposition~\ref{prop:twosubgroups} to get the condition $H_2^g=H_1$.

\begin{example}
Let $K=\CC$ and $G$ be the Heisenberg group over the finite ring~${\ZZ/n\ZZ}$. The group $G$ is finite nilpotent and is generated by the elements~$x$,~$y$, and~$z$ (see Subsection~\ref{subsect:Heis}). Take a primitive root of unity ${\zeta\in \CC}$ of degree~$n$. Define the subgroups~${H_1=\langle x,z\rangle}$ and ${H_2=\langle y,z\rangle}$ of $G$.

Define the characters~${\chi_i\colon H_i\to K^*}$, $i=1,2$, by the formulas
$$
\chi_1(x)=1\,,\qquad \chi_2(y)=1\,,\qquad \chi_1(z)=\chi_2(z)=\zeta\,.
$$
Then the subgroups $H_i\subset G$ are normal, the group $G/H_1$ is generated by the image of~$y$, the group $G/H_2$ is generated by the image of $x$, and there are equalities $\chi_1^{y^k}(x)=\zeta^k$, $\chi_2^{x^k}(y)=\zeta^{-k}$ for any integer~$k$. It follows from the isomorphism~\eqref{eq:end} that the representations~${\ind_{H_i}^G(\chi_i)}$ are Schur irreducible, whence they are irreducible, being complex representations of a finite group (cf. Theorem~\ref{theor:char}).

Furthermore, the set $H_1\backslash G/H_2$ has only one element, $H_1\cap H_2$ is the group generated by $z$, whence $\chi_1|_{H_1\cap H_2}=\chi_2|_{H_1\cap H_2}$. Thus Proposition~\ref{prop:twosubgroups} implies that the representations~${\ind_{H_i}^G(\chi_1)}$ and~${\ind_{H_2}^G(\chi_2)}$ are isomorphic.

On the other hand, the subgroups $H_1$ and $H_2$ are not conjugate as they have different images in the quotient over the commutator subgroup (note that the subgroups $H_1^*$ and $H_2^*$ coincide with $G$, thus they are trivially conjugate).
\end{example}

\medskip

\begin{remark}\label{remark:maxconj}
Combining Remark~\ref{remark:maxrank} and Proposition~\ref{prop:twosubgroups}, we obtain the following specification of Theorem~\ref{theorem:main}. Suppose that the field~$K$ is algebraically closed and uncountable. Let $(H',\chi')$ be a finite $\pi$-weight pair such that the torsion-free rank of $H'$ is maximal among all finite $\pi$-weight pairs. Then the conjugacy class of the subgroup $(H')^*\subset G$ does not depend on the choice of~$H'$. Moreover, any representative $D$ of this conjugacy class contains a subgroup $H\subset D$ of finite index such that $\pi\simeq\ind_H^G(\chi)$ for a character $\chi$ of~$H$.
\end{remark}

\subsection{Non-monomial irreducible representations}\label{subsect:nonmonom}

Berman, \v{S}araja~\cite{BS} and Segal~\cite[Theorems A, B]{Seg} have independently constructed non-monomial irreducible complex representations for an arbitrary finitely generated nilpotent group which is not abelian-by-finite.

The general case is reduced to the case of the Heisenberg group over the ring of integers. In this case, one constructs an irreducible representation which is not only non-monomial, but is also not finitely induced from any (irreducible) representation of a proper subgroup. For the sake of completeness, we sketch this construction following~\cite{Seg}.

\medskip

We shall use the notation and facts from Subsection~\ref{subsect:Heis}. Thus $G$ is the Heisenberg group over~$\ZZ$ and $c\in K$ is a non-zero element which is not a root of unity. We will need one more interpretation of the category of representations of $G$ such that $z$ acts by~$c$.

Let $A=R*\Gamma$ be the skew group algebra of the group $\Gamma$ with coefficients in~${R=K[t,t^{-1}]}$. Explicitly, $A$ is isomorphic to $R[\gamma,\gamma^{-1}]$ as an $R$-module and the product in $A$ is uniquely determined by the rule ${\gamma t=c\,t\gamma}$. Thus the \mbox{$K$-algebra} $A$ is non-commutative and the subring $R\subset A$ is not in the center of $A$ (in particular,~$A$ is not an $R$-algebra).

For short, by an $A$-module, we mean a left~\mbox{$A$-mo\-dule}. It is easily shown that a $\Gamma$-equivariant $R$-module is the same as an~\mbox{$A$-mo\-dule}.  Thus the category of representations of $G$ such that $z$ acts by $c$ is equivalent to the category of~\mbox{$A$-modules}. Indeed, the algebra $A$ is isomorphic to the quotient $K[G]/(z-c)$ of the group algebra~${K[G]}$.

\medskip

The group $\SL_2(\ZZ)$ acts on the Heisenberg group $G$ as follows. A matrix
$$
\alpha=\begin{pmatrix}
     p & q  \\
     r & s \\
\end{pmatrix}\in \SL_2(\ZZ)
$$
sends $x$ to $x^py^r$, sends $y$ to $x^qy^s$, and fixes $z$. Accordingly, $\alpha$ acts on the $K$-algebra~$A$ by the formula
$$
\alpha(\gamma)=\gamma^p t^r\,,\qquad \alpha(t)=\gamma^q t^s\,.
$$
Given an $A$-module $M$, by $M_{\alpha}$ denote the $A$-module such that $M_{\alpha}=M$ as a $K$-vector space and an element $a\in A$ acts on $M_\alpha$ as~$\alpha^{-1}(a)$. Equivalently, $M_{\alpha}\simeq A\otimes_{(A,\alpha)} M$, that is, $M_{\alpha}$ is the extension of scalars of $M$ with respect to the homomorphism of algebras~${\alpha\colon A\to A}$.

Note that $\alpha$ does not come from a $\Gamma$-equivariant automorphism of $R$, or, equivalently, of $\GG_m$, because $\alpha$ mixes $\gamma$ and $t$. This is the reason to introduce the algebra $A$.

\medskip

Now let $\pi$ be a representation of $G$ such that $z$ acts by $c$ and let $M$ be the corresponding $A$-module. Suppose that $\pi\simeq\ind_H^G(\rho)$ for a proper subgroup $H\subset G$ and a representation $\rho$ of $H$. We can assume that $H$ is a maximal subgroup of~$G$.

It follows that the index of $H$ in $G$ is a prime~${p\geqslant 2}$. Moreover, there is a matrix
\begin{equation}\label{eq:alpha}
\alpha=\begin{pmatrix}
     1 & i  \\
     0 & 1 \\
\end{pmatrix}\in\SL_2(\ZZ)\,,\qquad 0\leqslant i<p\,,
\end{equation}
such that $\alpha(H)$ is generated by $x^p$, $y$, and $z$. Let $B$ be the subalgebra in $A$ generated by $\gamma^p$ and $t$. We have the embeddings of rings
$$
R\subset B\subset A\,.
$$

Note that $H$ is isomorphic to the Heisenberg group and representations of~$H$ such that $z$ acts by $c$ correspond to $B$-modules. It follows that there is a $B$-module $N$ and an isomorphism of $A$-modules ${M_\alpha\simeq A\otimes_B N}$.

\medskip

All these reasonings lead to the following statement.

\begin{proposition}\label{prop:nonind}
Let~$M$ be an $A$-module such that for any $\alpha\in\SL_2(\ZZ)$ as in the equation~\eqref{eq:alpha}, the $A$-module $M_\alpha$ is not isomorphic to $A\otimes _B N$ for any $B$-module $N$. Let $\pi$ be the representation of $G$ that corresponds to~$M$. Then $\pi$ is not isomorphic to $\ind_H^G(\rho)$ for any proper subgroup $H\subset G$ and any representation $\rho$ of~$H$.
\end{proposition}

By $F$ denote the field of fractions of $R$, that is, $F$ is the field $K(t)$ of rational functions on $\GG_m$.

\begin{remark}\label{remark:dimen}
Given a $B$-module $N$, consider $A\otimes_B N$ as an $R$-module. There is an isomorphism of $R$-modules (cf. the isomorphism~\eqref{eq:mackey-norm})
$$
A\otimes_B N\simeq \bigoplus_{i=0}^{p-1}N_{\gamma^i}\,,
$$
where $N_{\gamma^i}=N$ as a $K$-vector space and an element $f\in R$ acts on $N_{\gamma^i}$ as~${\gamma^{-i}(f)}$.
In particular, the dimension of the $F$-vector space $F\otimes_R(A\otimes_B N)$ is either infinite or divisible by $p$.
\end{remark}

\medskip

Now let us construct an irreducible $A$-module that satisfies the assumption of Proposition~\ref{prop:nonind}. Consider a twisted action of $\Gamma$ on~$F$ given by the formula
$$
\gamma\;:\;f(t)\longmapsto (t-1)f(ct)\,.
$$
Let $M$ be the $\Gamma$-equivariant $R$-submodule in $F$ generated by the constant function $1$. One easily checks that $M$ consists of all rational functions on $\GG_m$ that have poles of order at most one at the points $c^i$, $i<0$, and are regular elsewhere (note that $i$ runs over negative integers only). Also, by construction, we have an isomorphism of $A$-modules $M\simeq A/(\gamma-t+1)$.

For any $R$-submodule $L\subset M$, we have that $M/L$ is a torsion $R$-module and its support on $\GG_m$ is contained in the set $\{c^i\}_{i<0}$. Therefore the support is not invariant under the action of $\Gamma$ on $\GG_m$ unless it is empty. This proves the $M$ is an irreducible $\Gamma$-equivariant $R$-module.

\medskip

Further, let $\alpha$ be as in the equation~\eqref{eq:alpha}. Then $\alpha(\gamma)=\gamma$ and~${\alpha(t)=\gamma^it}$. It follows that $M_\alpha$ is isomorphic to the $A$-module
$$
A/(\gamma-\gamma^it+1)=A/(\gamma^i-\gamma t^{-1}-t^{-1})\,.
$$
This implies that the dimension of the $F$-vector space $F\otimes_R M_\alpha$ is equal to~$i$. Since $0\leqslant i<p$, by Remark~\ref{remark:dimen}, we see that the $A$-module $M_\alpha$ is not isomorphic to $A\otimes_B N$ for any $B$-module $N$. Thus $M$ satisfies the assumption of Proposition~\ref{prop:nonind}.

\medskip

We have shown that there is an irreducible (possibly, complex) representation of the Heisenberg group over~$\ZZ$ that is not induced from a representation of any proper subgroup.


\begin{thebibliography}{99}
\bibitem{Parshin-Arnal}
S.\,A.\,Arnal, A.\,N.\,Parshin, {\it On irreducible representations of discrete Heisenberg groups}, Mat. Zametki, {\bf 92}:3 (2012), 323--330; translation in Math. Notes {\bf 92}:3 (2012), 295--301.

\bibitem{Baumslag}
G.\,Baumslag, {\it Lecture notes on nilpotent groups}, Regional Conference Series in Mathematics, {\bf 2} (1971), American Mathematical Society, Providence, R.I.

\bibitem{BK}
S.\,D.\,Berman, E.\,S.\,Kerer, {\it Representations of a torsion-free nilpotent group of class two with two generators}, Funktsional. Anal. i Prilozhen., {\bf 11}:4 (1977), 70--71; translation in Funct. Anal. Appl., {\bf 11}:4 (1977), 301--302.

\bibitem{BS}
S.\,D.\,Berman, V.\,V.\,\v{S}araja, {\it Irreducible complex representations of finitely generated nilpotent groups}, Ukrain. Mat. Z., {\bf 29}:4 (1977), 435--442; translation in Ukrainian Math.~J., {\bf 29}:4 (1977), 331--337.

\bibitem{BZ}
J.\,N.\,Bernstein, A.\,V.\,Zelevinskii, {\it Representations of the group $GL(n,F)$, where $F$ is a local non-Archimedean field}, Uspehi Mat. Nauk, {\bf 31}:3(189) (1976), 5--70; translation in Russian Math. Surveys, {\bf 31}:3 (1976), 1--68.

\bibitem{Bernst}
J.\,Bernstein, {\it Draft of: representations of $p$-adic groups}, lectures at Harvard University written by K.\,Rumelhart, available at http://www.math.tau.ac.il/$\sim$bernstei\,.

\bibitem{Brown}
I.\,D.\,Brown, {\it Representation of finitely generated nilpotent groups}, Pacific J. Math., {\bf 45}:1 (1973), 13--26.

\bibitem{Dix}
J.\,Dixmier, {\it Sur les representations unitaires des groupes de Lie nilpotents. I}, Amer. J. Math., {\bf 81} (1959), 160--170.

\bibitem{JS}
J.\,Jacobsen, H.\,Stetk{\ae}r, {\it Ultra-irreducibility of induced representations}, Math. Scand., {\bf 68}:2 (1991), 305--319.

\bibitem{Kir59}
A.\,A.\,Kirillov, {\it Induced representations of nilpotent Lie groups}, Dokl. Akad. Nauk SSSR, {\bf 128} (1959), 886--889.

\bibitem{Kirillov}
A.\,A.\,Kirillov, {\it Unitary representations of nilpotent Lie groups}, Uspehi Mat. Nauk, {\bf 17}:4(106) (1962), 57--110; translation in Russian Math. Surveys, {\bf 17}:4 (1962), 53--104.

\bibitem{Mac}
G.\,W.\,Mackey, {\it On induced representations of groups}, Amer. J. Math., {\bf 73}:3, (1951), 576--592.

\bibitem{Mal}
A.\,I.\,Malcev, {\it On a class of homogeneous spaces}, Izv. Akad. Nauk SSSR Ser. Mat., {\bf 13}:1 (1949), 9--32; translation in Amer. Math. Soc. Translation, {\bf 1951}:39 (1951).

\bibitem{Man}
A.\,Mann, {\it How groups grow}, London Mathematical Society Lecture Note Series, {\bf 395} (2012), Cambridge University Press, Cambridge.

\bibitem{Parshin}
A.\,N.\,Parshin, {\it On holomorphic representations of discrete Heisenberg groups}, Funktsional. Anal. i Prilozhen., {\bf 44}:2 (2010), 92--96; translation in
Funct. Anal. Appl., {\bf 44}:2 (2010), 156--159.

\bibitem{ParshCong}
A.\,N.\,Parshin, {\it Representations of higher adelic groups and arithmetic}, Proceedings of the International Congress of Mathematicians, {\bf I} (2010), 362--392, Hindustan Book Agency, New Delhi.

\bibitem{Seg}
D.\,Segal, {\it Irreducible representations of finitely generated nilpotent groups}, Math. Proc. Cambridge Philos. Soc., {\bf 81}:2 (1977), 201--208.

\bibitem{Serre}
J.-P.\,Serre,  {\it Linear representations of finite groups}, Graduate Texts in Mathematics, {\bf 42}, Springer-Verlag (1977).

\bibitem{Vigneras}
M.-F.\,Vign\'eras, {\it Repr\'esentations l-modulaires d'un groupe r\'eductif p-adique avec $l \ne p$}, Progress in Mathematics, {\bf 137} (1996), Birkh\"auser Boston, Inc., Boston, MA.

\end{thebibliography}
\end{document}